\documentclass[a4paper, 12pt]{amsart}
\usepackage{cite}
\usepackage{amsmath}
\usepackage{amsthm}
\usepackage[all]{xy}
\usepackage{amssymb}
\usepackage{mathrsfs}
\usepackage{etoolbox}
\usepackage{fullpage}
\newtheorem{thm}{Theorem}
[section]

\newtheorem{con}[thm]{Conjecture}

\newtheorem{prop}[thm]{Proposition}

\theoremstyle{definition}

\newtheorem{fct}[thm]{Fact}
\newtheorem{defn}[thm]{Definition}

\newtheorem{rmk}[thm]{Remark}

\numberwithin{equation}{subsection}
\newtheorem*{prop1}{Proposition}

\newcommand{\gen}[1]{\left\langle#1\right\rangle}
\usepackage{color}
\begin{document}

\title{Geometric triviality of the Strongly Minimal\\second Painlev\'e equations.}
\author{Joel Nagloo$^1$ \\University of Leeds}
\date{\today}
\pagestyle{plain}
\begin{abstract}
We show that for $\alpha\not\in 1/2+\mathbb{Z}$, the second Painlev\'e equation $P_{II}(\alpha):\;y'' = 2y^3+ty+\alpha$ is geometrically trivial, that is we show that if $y_{1},...,y_{n}$ are distinct solutions such that $y_{1},y_{1}',y_{2},y_{2}',\ldots,y_{n},y_{n}'$ are algebraically dependent over $\mathbb{C}(t)$, then already for some $1\leq i < j \leq n$, $y_{i},y_{i}',y_{j},y_{j}'$ are algebraically dependent over $\mathbb{C}(t)$. This extend to the non generic parameters the results in \cite{NagPil} for $P_{II}(\alpha)$.
\end{abstract}
\maketitle 

\footnotetext[1]{Supported by an EPSRC Project Studentship and a University of Leeds - School of Mathematics partial scholarship}

%-----------------------------------------------------------------------------------------------------------------------------------------
\section{Introduction}

In this paper, we initiate the study of existence of algebraic relations over $\mathbb{C}(t)$ between solutions of the second Painlev\'e equation $P_{II}(\alpha):\;y'' = 2y^3+ty+\alpha$ for non generic $\alpha\in\mathbb{C}$. In the case of generic $\alpha\in\mathbb{C}$, the work in \cite{NagPil1} gives a complete answer:
\begin{prop1} Suppose $\alpha\not\in\mathbb{Q}^{alg}$. Then if $y_1,\ldots,y_n$ are solutions of $P_{II}(\alpha)$ then  $y_1,y'_1,\ldots,$ $y_n,y'_n$ are algebraically independent over $\mathbb{C}(t)$, that is $tr.deg\left(\mathbb{C}(t)(y_1,y'_1,\ldots,y_n,y'_n)/\mathbb{C}(t)\right)=2n$. 
\end{prop1}
One important step in the proof of that result was to first show that in the generic case, $P_{II}(\alpha)$ is geometrically trivial \cite{NagPil}. In this paper we have succeeded to prove that the same hold when $\alpha\not\in 1/2+\mathbb{Z}$. We say a few words about our strategy:\\

It is well known that for $\alpha\not\in 1/2+\mathbb{Z}$, $P_{II}(\alpha)$ is strongly minimal: if $y$ is a solution which satisfies an algebraic differential equation of the first order over a differential field extension $K$ of $\mathbb{C}(t)$, then $y$ is algebraic over $K$. On the other hand, from model theory, we have a very general classification of differential equations that are strongly minimal. For $P_{II}(\alpha)$, this classification says that either 1) it is geometrically trivial or 2) there is a differential rational relation between $P_{II}(\alpha)$ and a very special kind of ODE, $E^{\sharp}$, on an elliptic curve $E$ not defined over $\mathbb{C}$ (an example being the Picard-Painlev\'e VI).\\

Inspired by the work of Nishioka in \cite{Nishioka}, we show that for $\alpha\not\in 1/2+\mathbb{Z}$ any finite-to-finite correspondence $P_{II}(\alpha)\rightarrow P_{II}(\alpha)$ is an automorphism. From this we can conclude that $P_{II}(\alpha)$ is geometrically trivial. Indeed, on the $E^{\sharp}$'s are defined the multiplication-by-$n$ maps ($n\in\mathbb{N}$), and we show that these correspondences must be ``carried through" via any differential rational relation.\\

The paper is organized as follows. In Section 2 we set up the notation and recall some results we need from the study of strongly minimal sets in differentially closed fields of characteristic $0$. We end that section by proving that if $Y$ is a non-trivial locally modular strongly minimal set, then over some differential field $K$, there exits finite-to-finite correspondences $Y\rightarrow Y$ that are not automorphisms of $Y$ (Proposition \ref{prop1}). Using this, in Section 3 we show that for $\alpha\not\in 1/2+\mathbb{Z}$, $P_{II}(\alpha)$ cannot be non-trivial locally modular and so must be geometrically trivial (Proposition \ref{Mainprop}). Finally in Section 4 we give a restatement of the conjecture that in $DCF_0$ any geometrically trivial strongly minimal set is $\omega$-categorical.

%Finally in this introduction some acknowledgements: 
\subsection*{Acknowledgement}
The author would like to thank Prof. Anand Pillay for his support, guidance and continuous encouragement; but also for supplying the proof of Proposition \ref{prop1}. Thanks also to James Freitag for an interesting discussion around some of the problems considered in this paper.
%-----------------------------------------------------------------------------------------------------------------------------------------

\section{Preliminaries}

%We should make it clear that $\mathbb C$ can be seen in two ways, as a subfield of ${\mathcal U}$ (over which varieties, differential algebraic varieties,..  may be defined) and as a differential algebraic variety (or definable set) in its own right, defined by $y' = 0$.

We start with a brief summary of the notions from model theory and differential algebra that we will need (cf. \cite{Marker}). We fix the language $L_{\partial}$=($+,-,\cdot,0,1,\partial$) of differential fields. By $DCF_0$, we mean the theory of differentially closed fields of characteristic $0$ with a single derivation. It is well known that that $DCF_0$ is complete, has quantifier elimination and is $\omega$-stable. Although not explicit, these properties play an important role throughout this paper. For us, a differentially closed field is precisely a model of $DCF_0$. For a field $F$, $F^{alg}$ will denote its algebraic closure in the usual algebraic sense and by a definable set we mean a finite Boolean combinations of affine differential algebraic varieties (or Kolchin closed sets). If a definable set $Y$ in $K^n$ is defined with parameters from a differential subfield $F$ of $K$ we will say $Y$ is defined over $F$.\\

%A remark is that if $(F,\partial)$ is a {\em differential} field (of characteristic $0$) then $\partial$ extends uniquely to a derivation on $F^{alg}$  by $L_{\partial}$-formulas

We will also fix a ``saturated" model $\mathcal U = ({\mathcal U},+,-,\cdot,0,1,\partial)$ of $DCF_{0}$ of cardinality continuum and so we take $\mathbb{C}$, the field of complex numbers, as the field of constants of $\mathcal U$. Throughout $t$ will denote an element of $\mathcal{U}$ such that $\partial(t)=1$. If $K$ is a differential subfield of ${\mathcal U}$ and $y$ is a tuple from ${\mathcal U}$, then $K\langle y \rangle$ will denote the differential field $K(y,y',y'',...)$, where here $y'$ is short for $\partial(y)$. As usual, given a tuple $z$ from ${\mathcal U}$, $z\in K\langle y \rangle^{alg}$ means that the coordinates of $z$ are in $K\langle y \rangle^{alg}$.\\

Given a differential field $K$ and $\overline{a}$ a tuple of elements from $\mathcal{U}$, the type of $\overline{a}$ over $K$, denoted $tp(\overline{a}/K)$, is the set of all $L_{\partial}$-formulas with parameters from $K$ that $\overline{a}$ satisfies. It is not hard to see that the set $I_{p}=\{f\in K\{\overline{X}\}: f(\overline{X})=0\in p\}=\{f\in K\{\overline{X}\}: f(\overline{a})=0\}$ is a differential prime ideal in $K\{\overline{X}\}=K[\overline{X},\overline{X}',\ldots]$, where $p=tp(\overline{a}/K)$. Indeed, by quantifier elimination, the map $p\mapsto I_p$ is a bijection between the set of (complete) types over $K$ and differential prime ideals in $K\{\overline{X}\}$. Therefore in what follows there is no harm to think of $p=tp(\overline{a}/K)$ as the ideal $I_{p}$.\\

We will say that a definable set $Y\subseteq {\mathcal U}^{n}$ is finite dimensional if $order(Y)=sup\{$ $tr.deg(K\langle y \rangle/K):y\in Y\}$ is finite,  where $K$ is some differential field over which $Y$ is defined. For such a $Y$, we will call an element $y\in Y$ {\em generic over $K$} if $tr.deg(K\langle y\rangle/K) = order(Y)$.\\

\noindent An important class of finite dimensional sets is the following:

\begin{defn} A definable set $Y$ in ${\mathcal U}^{n}$ is {\em strongly minimal} if it is infinite but has no infinite co-infinite definable subsets. %Equivalently, if it cannot be written as the disjoint union of definable sets of order $m$ (where $m$ is the order of $X$), and for any differential field $K$ over which $Y$ is defined, and point $y\in Y$, either $y\in K^{alg}$ or $y$ is generic point of $Y$ over $K$ (i.e. $tr.deg(K\langle y \rangle/K)$ is $0$ or $m$).
\end{defn}

In \cite{NagPil} one can find a very detailed summary of the main results around the ``geometry" of strongly minimal sets. Here we recall the things we need and leave the details.

\begin{rmk}\label{sm-painleve}
If $Y\subseteq {\mathcal U}$ is defined by differential equation $y^{(n)} = f(y,y',..,y^{(n-1)},t)$ where $f$ is rational over $\mathbb C$, then $Y$ is strongly minimal if and only if for any differential subfield $F$ of $\mathcal U$ which is finitely generated over ${\mathbb C}(t)$, and $y\in Y$, either $y\in F^{alg}$ or $tr.deg(F\langle y \rangle/F) = n$. 
\end{rmk}

The first basic example of a strongly minimal set in $\mathcal{U}$ is the field of constants $\mathbb{C}=\{y\in\mathcal{U}:y'=0\}$. Other important examples come from simple abelian varieties $A$  over ${\mathcal U}$:
\begin{fct}[\cite{Buium},\cite{Hrushovski-Sokolovic}] Let $A$ be an abelian variety over ${\mathcal U}$. We identify $A$ with its set $A({\mathcal U})$ of ${\mathcal U}$-points. Then
\newline
(i) $A$ has a (unique) smallest Zariski-dense definable subgroup, which we denote by $A^{\sharp}$. 
\newline
%(ii) $A^{\sharp}$ is finite-dimensional, $dim(A)\leq order(A^{\sharp}) \leq 2dim(A)$ and moreover $dim(A) = order(A^{\sharp})$ if and only if $A$ descends to $C_{\mathcal{U}}$ (in which case $A^{\sharp} = A(C_{\mathcal{U}})$), 
%\newline
(ii) If $A$ is a simple abelian variety with $\mathbb{C}$-trace $0$, then $A^{\sharp}$ is strongly minimal. 
%\newline
%(iv) If $A$ is an elliptic curve then $A^{\sharp}$ is strongly minimal, whether or not $A$ descends to $C_{\mathcal{U}}$.
\end{fct}
The subgroup $A^{\sharp}$ is called the Manin kernel of $A$ (cf. \cite{Marker-Quaderni}) and it will play a very important role in what follows. 

\begin{defn} Let $Y\subset {\mathcal U}^{n}$ be strongly minimal, and suppose $order(Y) = m$.
\newline
(i) We say that $Y$ is {\em geometrically trivial} if for any countable differential field $K$ over which $Y$ is defined, and for any $y_{1},..,y_{\ell}\in Y$, if the collection consisting of $y_{1},..,y_{\ell}$ together with all their derivatives $y_{i}^{(j)}$ is algebraically dependent over $K$ then for some $i<j$, $y_{i}, y_{j}$ together with their derivatives is algebraically dependent over $K$.
\newline
(ii) Let $Z$ be another strongly minimal set and denote by $\pi_1:Y\times Z\rightarrow Y$ and $\pi_2:Y\times Z\rightarrow Z$ the projections to $Y$ and $Z$ respectively. We say that $Y$ and $Z$ are nonorthogonal if there is some infinite definable relation $R\subset Y\times Z$ such that ${\pi_1}_{\restriction R}$ and ${\pi_2}_{\restriction R}$ are finite-to-one functions.  
\end{defn}

\begin{rmk}\label{remorth} Suppose $Y$ and $Z$ are nonorthogonal strongly minimal sets and that the relation $R\subset Y\times Z$ is defined over some field $K$. Then by definition, for any generic $y\in Y$ there exist $z\in Z$ generic such that $(y,z)\in R$ and in that case $K\gen{y}^{alg}=K\gen{z}^{alg}$. So it is not hard to see that, if $Y,Z$ are nonorthogonal strongly minimal sets then $order(Y) = order(Z)$.
\end{rmk}

Nonorthogonality is an equivalence relation for strongly minimal sets and the following theorem, called the trichotomy, gives a classification of strongly minimal set in $\mathcal{U}$ up to nonorthogonality. We direct the reader to \cite{NagPil} for a summary of how the proof goes.
\begin{thm}[\cite{Hrushovski-Sokolovic}]\label{trichotomy} Let $X$ be a strongly minimal set. Then exactly one of the following holds:
\newline
(i) $X$ is nonorthogonal to the strongly minimal set $\mathbb{C}$ (defined by $y' = 0$),
\newline
(ii) $X$ is nonorthogonal to $A^{\sharp}$ for some simple abelian variety $A$ over ${\mathcal U}$ which does not descend to $\mathbb{C}$. 
\newline
(iii) $X$ is geometrically trivial. 
\end{thm}
Sets in (ii) are called {\em non-trivial locally modular} strongly minimal sets. It is worth mentioning that without any doubt, the trichotomy (as stated) is one of the deepest result in the model theory of differentially closed field of characteristic $0$. In any case, in this paper we will use it to make a crucial observation about the non-trivial locally modular strongly minimal sets.\\

Let $Y$ be a strongly minimal set in $\mathcal{U}$ and suppose $K$ is any differential field over which $Y$ is defined. If $y$ and $z$ in $Y$ are such that $y\in K\gen{z}^{alg}$ (and so $z\in K\gen{y}^{alg}$), then $mult(y/K\gen{z})$ will denote the (finite) number of elements of $Y$ that realizes $tp(y/K\gen{z})$. In other words, $mult(y/K\gen{z})$ is the degree of the minimal polynomial of $y$ over $K\gen{z}$.\\

It is well known that finite-to-finite correspondences (indeed finite-to-one) exist for the Manin kernels: Let $A$ be a simple abelian variety with $\mathbb{C}$-trace $0$ and let $Y=A^{\sharp}$. For each $n\in\mathbb{N}$, we have multiplication-by-$n$ map, $n:A\rightarrow A$. This map is surjective and $n^{2d}$-to-$1$ (where $d$ is the dimension of $A$). So for any generic point $a\in Y$, $mult(a/n\cdot a))=1$ while $mult(n\cdot a/a)>1$ (as long as $n>1$).\\

Our first result says that the same is true for any non-trivial locally modular strongly minimal sets:
\begin{prop}\label{prop1}
Suppose that $Y$ is a non-trivial locally modular strongly minimal set in $\mathcal{U}$. Then there exists a differential field $K$ such that for any generic point $y$ of $Y$ over $K$, there exist $z\in Y$ generic over $K$, such that $z\in K\gen{y}^{alg}$ and $z\not\in K\gen{y}$.
\end{prop}
\begin{proof}
Assume for contradiction that the conclusion of the proposition does not hold.\\
Now as $Y$ is non-trivial locally modular, from Theorem \ref{trichotomy}, there exist a simple Abelian variety $A$ not defined over $\mathbb{C}$ such that $Y$ is nonorthogonal to $X=A^{\sharp}$. Let $K$ be the differential field over which this occur and throughout will be working over $K$. Also, in what follows $Y_{gen}$ will be the set of generic points of $Y$ (same for $X_{gen}$).\\
We show that for $d,e\in X_{gen}$, if $d$ and $e$ are interalgebraic, then $mult(e/d)=mult(d/e)$. This contradicts the case when $d=n\cdot e$ (for any $n\in\mathbb{N}_{>1}$) as discussed above and we are done. On the other hand as $Y$ and $X$ are nonorthogonal, if we let $y\in Y_{gen}$ and let $Z=acl(y)\cap(Y_{gen}\cup X_{gen})$, then it is enough to show that if $d,e\in Z\cap X$, then $mult(e/d)=mult(d/e)$. We prove this using several claims.\\

Let $a\in Z\cap Y$ and $d,e\in Z\cap X$.\\
{\bf Claim 1:}  $mult(e/a)=mult(d/a)$ and $mult(a/e)=mult(a/d)$. Indeed $mult(d/a)$ and $mult(a/d)$ does not depend on the choice of $a$ or $d$.\\
{\em Proof:} First note that by construction (and nonorthogonality) $a$,$d$ and $e$ are all interalgebraic. Indeed any two elements of $Z$ are.\\
We first prove that $mult(e/a)=mult(d/a)$. So suppose $D=\{d=d_1,\ldots,d_k\}$ is the set of realizations of $tp(d/a)$, i.e $mult(d/a)=k$. Now as $tp(e)=tp(d)$, it is not hard to see that there is $b\in Y_{gen}$ and $E=\{e=e_1,\ldots,e_k\}\subseteq X_{gen}$ such that $E$ is the set of realizations of $tp(e/b)$. But then $K\gen{a}^{alg}=K\gen{b}^{alg}$ and hence by assumption $K\gen{a}=K\gen{b}$. So for any $\sigma\in Aut(\mathcal{U}/a)$, $\sigma(b)=b$ and hence $\sigma(E)=E$. In particular $E$ is definable over $a$ and we have that $E$ is the set of realizations of $tp(e/a)$, i.e $mult(e/a)=k$.\\
We now prove that $mult(a/e)=mult(a/d)$. Suppose $mult(a/d)=l$ and that $\phi(x,d)$ isolates $tp(a/d)$, i.e $\models\exists^{=l}y\phi(y,d)$. As $tp(e)=tp(d)$, we have that $\models\exists^{=l}y\phi(y,e)$. Choose $c\in Y_{gen}$ such that $\models\phi(c,e)$. Then $K\gen{a}^{alg}=K\gen{c}^{alg}$ and hence by assumption $K\gen{a}=K\gen{c}$, that is there is a $L_{\partial}$-formula $\theta(x,y)$ such that $\models\theta(a,c)\wedge\exists^{=1}x\theta(x,c)\wedge\exists^{=1}y\theta(a,y)$. Let $\psi(x,e)$ be the $L_{\partial}$-formula $\exists y\phi(y,e)\wedge\theta(x,y)$. By construction $\models\exists^{=l}x\psi(x,e)$ and it follows that $\psi(x,e)$ isolates $tp(a/e)$, i.e $mult(a/e)=l$.\\

\noindent{\bf Claim 2:} $mult(d/a,e)=mult(e/a,d)$.\\
{\em Proof:} Since 
\begin{eqnarray*}
mult(de/a)&=&mult(e/a)\cdot mult(d/a,e)\\
 &=& mult(d/a)\cdot mult(e/a,d),
\end{eqnarray*}
using Claim 1 we are done.\\

\noindent{\bf Claim 3:} $mult(da/e)=mult(ea/d)$.\\
{\em Proof:} As before, since
\begin{eqnarray*}
mult(da/e)&=&mult(a/e)\cdot mult(d/e,a)\\
mult(ea/d)&=&mult(a/d)\cdot mult(e/a,d),
\end{eqnarray*}
using both Claim 1 and Claim 2 we are done.\\

\noindent Finally\\
{\bf Claim 4:} $mult(d/e)=mult(e/d)$.\\
{\em Proof:} This time we use
\begin{eqnarray*}
mult(da/e)&=&mult(d/e)\cdot mult(a/d,e)\\
mult(ea/d)&=&mult(e/d)\cdot mult(a/d,e).
\end{eqnarray*}
So from Claim 3 the result follows
\end{proof}

So if $Y$ is a non-trivial locally modular strongly minimal set, over some differential field $K$, there exits finite-to-finite correspondences $Y\rightarrow Y$ that are not automorphisms of $Y$. The aim of the next section will hence be to show that such is not the case for the the strongly minimal second Painlev\'e equations.

\section{Geometric triviality and the second Painlev\'e equations} 
In this section we look at the second Painlev\'e equation, $y''=2y^3+ty+\alpha$, $\alpha\in\mathbb{C}$ and denote by $Y(\alpha)\subseteq\mathcal{U}$ its solution set. From \cite{NagPil} we have the following:
\begin{prop}\label{prop2}{\color{white} hs}\\
(i) $Y(\alpha)$ is strongly minimal if and only if $\alpha\not\in\frac{1}{2}+\mathbb{Z}$.\\
(ii) For $\alpha\not\in\mathbb{Q}^{alg}$, $Y(\alpha)$ is geometrically trivial.
\end{prop}
We now aim to extend Proposition \ref{prop2}(ii) to all $\alpha\not\in\frac{1}{2}+\mathbb{Z}$.
\begin{rmk}\label{Rem}
Strong minimality of $Y(\alpha)$ is also equivalent to the following statement (cf. \cite{NagPil}): Let $y\in Y(\alpha)$ with $y\not\in K^{alg}$ ($K$ some differential field over which $Y(\alpha)$ is defined) and consider the polynomial algebra $K[y,y']$. If a nonzero polynomial $f\in K[y,y']$ divides it derivative $f'$, then $f\in K$.\\
This is what is often called ``Condition J" of Umemura.
\end{rmk}
Before we proceed recall that for a field $K$,  $K\left(\left(X\right)\right)$ denotes the field of formal Laurent series in variable $X$, while $K\gen{\gen{X}}$ denotes the field of formal Puiseux series, i.e. the field $\bigcup_{d\in{\mathbb{N}}}K\left(\left(X^{1/d}\right)\right)$. It is well know that if $K$ is algebraically closed then so is $K\gen{\gen{X}}$ (cf. \cite{Eisenbud}).\\%In Eisenbud corollary 13.15 (p. 295)

\begin{prop}\label{Mainprop}
For $\alpha\not\in\frac{1}{2}+\mathbb{Z}$, $Y(\alpha)$ is geometrically trivial.
\end{prop}
\begin{proof}
We fix $\alpha\not\in\frac{1}{2}+\mathbb{Z}$. First note that by Remark \ref{remorth}, $Y(\alpha)$ is orthogonal to $\mathbb{C}$. So from Proposition \ref{prop1} we only have to prove that if $y$ and $z$ are two generic elements of $Y(\alpha)$, then for any differential field $K$ over which $Y(\alpha)$ is defined, if $K(y,y')^{alg}=K(z,z')^{alg}$, then $K(y,y')=K(z,z')$.\\

So let $K$ be any differential field containing $\mathbb{C}(t)$ and let $y,z\in Y(\alpha)$ (generic) be such that $z\in K(y,y')^{alg}$. Letting  $K_1$ denote the algebraic closure of $K(y')$ in $\mathcal{U}$, we regard $z$ as algebraic over $K_1(y)$. Then for any  $\beta\in K_1$, $z$ can be seen as an element of $K_1\gen{\gen{y-\beta}}$, so that there exists $e\in\mathbb{N}$ such that $z\in K_1\left(\left((y-\beta)^{1/e}\right)\right)$. A simpler way of saying the above ($\grave{a}${\em -la-Nishioka}) is that we look at expansions in a local parameter $\tau$ at $\beta\in K_1$ given by
\[y=\beta+\tau^e\;\;\;\;\;\;\;\;\;\;\;\;\;z=\sum_{i=r}^{\infty}a_i\tau^i\;\;\;\;(a_r\neq0)\]
with $e$ the ramification exponent.
\par Differentiating we have
\begin{eqnarray*}
e\tau^{e-1}\tau' &=& y'-\beta'\\
	&=& y'-\beta^*-\beta_{y'}(2y^3+ty+\alpha)\\
	%&=& y'-y^2-\frac{t}{2}-\beta^*-\beta_y(2xy+\alpha+1/2)\\
	%&=& y-(\beta+\tau^e)^2-\frac{t}{2}-\beta^*-\beta_y(2(\beta+\tau^e)y+\alpha+1/2)\\
	&=& y'-\beta^*-\beta_{y'}(2\beta^3+t\beta+\alpha)+(6\beta^2+t)\beta_{y'}\tau^e+6\beta\beta_{y'}\tau^{2e}+\beta_{y'}\tau^{3e}
\end{eqnarray*}
where ``$^*$'' indicates the extension of the derivation on $K[y']$ given by $(\sum c_i{y'}^i)^*=\sum c'_i{y'}^i$ and $\beta_{y'}=\frac{\partial\beta}{\partial y'}$.
\par Letting $\gamma=y'-\beta^*-\beta_{y'}(2\beta^3+t\beta+\alpha)$, we have\\
{\bf Claim 1:} $\gamma\neq0$\\
{\em Proof:} For contradiction, suppose that $\gamma=0$, that is
\begin{equation}\label{gamma}
y'-\beta^*=\beta_{y'}(2\beta^3+t\beta+\alpha)
\end{equation}
Let $F\in K[y,y']$ be an irreducible polynomial such that $F(\beta,y')=0$. Then
\begin{eqnarray}
F^*(\beta,y')+\beta^*F_y(\beta,y')&=&0\\ \label{eqn1}
F_{y'}(\beta,y')+\beta_{y'}F_y(\beta,y')&=&0. \label{eqn2}
\end{eqnarray}
If we multiply \ref{eqn2} by $2\beta^3+t\beta+\alpha$ and use \ref{gamma}, we have
\[(2\beta^3+t\beta+\alpha)F_{y'}(\beta,y')+(y-\beta^*)F_y(\beta,y')=0.\] So, together with \ref{eqn1} we get
\[F^*(\beta,y')+y'F_y(\beta,y')+(2\beta^3+t\beta+\alpha)F_{y'}(\beta,y')=0.\]
In other words $F'(\beta,y)=0$, so that $F$ divides its derivative $F'$. This contradicts strong minimality of $Y(\alpha)$ as per Remark \ref{Rem} and the claim is proved.\\
\par Hence $\gamma\neq0$ implies that $\tau'=e^{-1}\gamma\tau^{1-e}+\cdots$ and from this we get that
\begin{eqnarray*}
(\tau^i)'&=& i\tau^{i-1}\tau'\\
	&=& i\tau^{i-1}e^{-1}\gamma\tau^{1-e}+\cdots\\
	&=& \frac{i\gamma}{e}\tau^{i-e}+\cdots
\end{eqnarray*}
and similarly
\[(\tau^i)''=\frac{i(i-e)}{e^2}\gamma^2\tau^{i-2e}+\cdots.\]
Now from
\[z''=2z^3+tz+\alpha\;\;\;\;\;\;\;\;\text{and}\;\;\;\;\;z=\sum_{i=r}^{\infty}a_i\tau^i\]
we have 
\begin{equation}\label{main}
\sum_{i=r}^{\infty}a_i\frac{i(i-e)}{e^2}\gamma^2\tau^{i-2e}+\cdots=2\sum_{i,j,k}a_ia_ja_k\tau^{i+j+k}+\cdots.
\end{equation}
Using this we prove a couple of claims.\\
{\bf Claim 2:} If $e>1$ then $r<0$.\\
{\em Proof:} Let $e>1$ and assume $r\geq0$. Choose $l\in\{r,r+1,\ldots\}$ least such that $e\nmid l$ and $a_l\neq0$. First, one should note that since $l-2e<l$ and $e\nmid l-2e$, $a_{l-2e}=0$ (same for $a_{l-e}=0$), so that in what follows one does not need to worry about the other coefficients in \ref{main}.
\par So if we look at the coefficient of $\tau^{l-2e}$ on the LHS of \ref{main}, we see that
\[a_l\frac{l(l-e)}{e^2}\gamma^2\neq0.\]
This implies that the coefficient on the RHS of $\tau^{i+j+k}$ for some $i,j,k\geq r$ with $i+j+k=l-2e$ must be non-zero. However for any such, since $i+j+k<l$ and $e\nmid i+j+k$, we have that $e$ does not divide at least one them, say $i<l$. But then $a_i=0$ (as $l$ was chosen to be the least with this property) and so $a_ia_ja_k=0$, a contradiction.\\
%, that is the coefficient of $\tau^{l-2e}=\tau^{i+j+k}$ is zero on the RHS. C
{\bf Claim 3:} The case $e>1$ and $r<0$ leads to a contradiction.\\
{\em Proof:} So this time suppose $e>1$ and  $r<0$. From the least powers of $\tau$ in $\ref{main}$ we have $r-2e=3r$, that is $r=-e$, and from the coefficients of $\tau^{r-2e}$ we get
\[a_r\frac{r(r-e)}{e^2}\gamma^2=2a_r^3\]
so that $a_r=\pm\gamma$, since $a_r\neq0$.\\
So again choose $l\in\{r,r+1,\ldots\}$ least such that $e\nmid l$ and $a_l\neq0$. The coefficient of $\tau^{l-2e}=\tau^{l+2r}$ on the LHS of \ref{main} is 
\[a_l\frac{l(l-e)}{e^2}\gamma^2\neq0.\]
On the RHS we see that the coefficient of $\tau^{l-2e}=\tau^{l+2r}$ should be $6a_r^2a_l$. (Indeed, $e\nmid i+j+k$ means that $e$ does not divide at least one of them, say $i$. Then $e\nmid i$, $a_i\neq0$ means either $i=l$ or $i>l$. But $i>l$ implies that $j+k<2r$ a contradiction. So $i=l$ and hence $j=k=r$).\\
Hence
\[a_l\frac{l(l-e)}{e^2}\gamma^2=6a_r^2a_l=6\gamma^2a_l\]
and we see that either $l=-2e$ or $l=3e$, contradicting $e\nmid l$ and we are done\\

\par Hence $e=1$, and since $\beta$ was arbitrary, the ramification exponent at every $\beta\in K_1$ is 1. So $z\in K_1(y)$.
\par Finally, letting $K_0$ denote the algebraic closure of $K(y)$ in $\mathcal{U}$, one can show similarly that $z\in K_0(y')$. Since $K_1(y)\cap K_0(y')=K(y,y')$, we have shown that $z\in K(y,y')$.
\par Changing the role of $y$ and $z$, we also have $y\in K(z,z')$ and hence $K(y,y')=K(z,z')$.
\end{proof}

\section{Further comments}

Recall that a strongly minimal set $Y$ is said to be unimodular if for any $n$ and any differential field $K$ over which $Y$ is defined, if $y_1,\ldots,y_n$,$y'_1,\ldots,y'_n\in Y$ are such that $\overline{y}$ and $\overline{y}'$ are interalgebraic over K (i.e $K\gen{\overline{y}}^{alg}=K\gen{\overline{y}'}^{alg}$), and $RM(\overline{y}/K)=RM(\overline{y}'/K)=n$, then $mult(\overline{y}/K\gen{\overline{y}'})=mult(\overline{y}'/K\gen{\overline{y}})$. Here $RM(\overline{y}/K)=n$ means that $y_{1},..,y_{n}$ together with all their derivatives $y_{i}^{(j)}$ are algebraically independent over $K$.\\

It is not hard to check that since correspondences are automorphisms, the solution set $X(\alpha)$ of the second Painlev\'e equation is unimodular. On the other hand, we have seen that the Manin kernels are not unimodular. Hrushovski (\cite{private}) pointed out that for strongly minimal sets, unimodularity is preserved under nonorthogonality (and the proof of Proposition \ref{prop1} is basically a special case). So we could have also deduce in that way that $P_{II}(\alpha)$, $\alpha\not\in 1/2+\mathbb{Z}$, is geometrically trivial. More importantly though, we have that in $DCF_0$ any unimodular strongly minimal set is geometrically trivial. We conjecture that the converse is also true:
\begin{con}
In any differentially closed field, every geometrically trivial strongly minimal set is unimodular.
\end{con}
This is a restatement of the conjecture that in $DCF_0$ every geometrically trivial strongly minimal set is is infinite dimensional (cf. \cite{Marker2}). 
%----------------------------------------------------------------------------------------------------------------------------------------


\begin{thebibliography}{99}

%\bibitem{Works} Bass H., Buium, A., Cassidy, P., eds.
%{Selected works of Ellis Kolchin with commentary},
%Amer. Math. Soc., Providence, RI (1999).

\bibitem{Buium} A. Buium, {\em Differential Algebraic Groups of Finite Dimension}, Lecture Notes in Mathematics 1506, Springer 1992.

%\bibitem{Casale}  G. Casale, Une preuve Galoisienne de l'irr\'educibilit\'e au sens de Nishioka-Umemura de la premi\`ere \'equation de Painlev\'e, Asterisque 323, 2009. 

%\bibitem{Hrushovskilocallymodular} E. Hrushovski, Locally modular regular types, in {\em Classification Theory, Proceedings, Chicago 1985} (ed. J. Baldwin), Springer, 1987.

%\bibitem{Hrushovski-Itai} E. Hrushovski  and M. Itai,
%{On model complete differential fields.}
%Trans. Amer. Math. Soc., vol 355(11) (2003), 4267-4296.
\bibitem{Eisenbud}D. Eisenbud,
{Commutative Algebra with a View Toward Algebraic Geometry.}
Graduate Texts in Mathematics 150. Springer-Verlag, 1995.


%\bibitem{Hrushovski-Jouanalou} E. Hrushovski, 
%{ODEs of order 1 and a generalization of a theorem of Jouanalou.}
%unpublished manuscript, 1995.

\bibitem{Hrushovski-Sokolovic} E. Hrushovski and Z. Sokolovic,
{Strongly minimal sets in differentially closed fields},
 unpublished manuscript, 1994.

%\bibitem{Hrushovski-Zilber} E. Hrushovski and B. Zilber, Zariski geometries,  Journal AMS, 9 (1996), 1 - 56

\bibitem{private} E. Hrushovski
{Private communications.}
(2012)

\bibitem{Marker} D. Marker, 
Model theory of differential fields, in Model theory of fields, second edition, (Marker, Messmer, Pillay)  Lecture Notes in Logic 5, ASL-CUP, 2006.

\bibitem{Marker2} D. Marker,
{The Number of Countable Differentially closed fields}
Notre Dame Journal of Formal Logic, vol. 48, Number 1 (2007), 99-113.

\bibitem{Marker-Quaderni} D. Marker, Manin kernels, in {\em Connections between Model Theory and Algebraic and Analytic Geometry}, edited by A. Macintyre, Quaderni Matematica vol. 6, Naples II, 2000. 

\bibitem{NagPil}Nagloo J. and Pillay A.,
{On Algebraic relations between solutions of a generic Painlev\'e equation}
http://arxiv.org/abs/1112.2916, submitted.

\bibitem{NagPil1}Nagloo J. and Pillay A.,
{On the algebraic independence of generic Painlev\'e transcendents}
http://arxiv.org/abs/1209.1562, submitted.

\bibitem{Nishioka} K. Nishioka ,
{Algebraic independence of Painlev\'e first transcendents},
Funkcialaj Ekvacioj, 47 (2004), 351-360.

%\bibitem{Pillay-book} A. Pillay,
%{\em Geometric Stability Theory}, Oxford University Press, 1996.

%\bibitem{Pillay-DAG}  A. Pillay, Differential algebraic groups and the number of countable differentially closed fields, in Model Theory of Fields, second edition, (Marker, Messmer, Pillay) Lecture Notes in Logic 5, ASL-CUP, 2006.

%\bibitem{Pillay-Fields} A. Pillay, Differential Fields, in {\em Lecture Notes in Algebraic Model Theory} (ed. Hart, Valeriote), Fields Inst. Monograph 15, AMS (2002), 3-48.

%\bibitem{Pillay-Ziegler} A. Pillay and M. Ziegler, Jet spaces of varieties over differential and difference fields, Selecta Math. New series, 9 (2003), 579-599.

%\bibitem{Umemura2} H. Umemura  and H. Watanabe,
%{Solutions of the second and fourth Painlev\'e equations, I}
%Nogoya Maths., J 148 (1997), 151-198.

\end{thebibliography}
\end{document}